\documentclass[]{article}
\usepackage{amssymb,amsmath,amsthm,bm,citesort,setspace,mathrsfs}
\usepackage{url,color,units}
\input cyracc.def

\newcommand{\Z}{\mathbf{Z}}

\newcommand{\R}{\mathbf{R}}

\newcommand{\be}{\begin{equation}}
\newcommand{\ee}{\end{equation}}

\newcommand{\lip}{\text{\rm Lip}_\sigma }

\renewcommand{\P}{\mathrm{P}}
\newcommand{\E}{\mathrm{E}}
\newcommand{\F}{\mathcal{F}}

\newcommand{\1}{\boldsymbol{1}}
\renewcommand{\d}{{\rm d}}

\newcommand{\e}{{\rm e}}

\renewcommand{\leq}{\leqslant}
\renewcommand{\ge}{\geqslant}
\renewcommand{\le}{\leqslant}

\author{Daniel Conus\\Lehigh University
\and Mathew Joseph\\University of Utah
\and Davar Khoshnevisan\\University of Utah}
\title{Correlation-length bounds, and 
	estimates for intermittent islands in parabolic SPDEs\thanks{
	Research supported in part by 
	the NSFs grant DMS-0747758 (M.J.) and DMS-1006903 (D.K.).}}
\date{October 13, 2011}
\newtheorem{stat}{Statement}[section]

\newtheorem{corollary}[stat]{Corollary}
\newtheorem{theorem}[stat]{Theorem}
\newtheorem{lemma}[stat]{Lemma}
\theoremstyle{definition} 

\newtheorem{remark}[stat]{Remark}

\numberwithin{equation}{section}
\begin{document}
\maketitle
\begin{abstract}

We consider the nonlinear stochastic heat equation in one dimension. 
Under some conditions on the nonlinearity, we show that the ``peaks" 
of the solution are rare, almost fractal like. 
We also provide an upper bound on the length of the ``islands,''
the regions of large values. These results are obtained by analyzing the 
\textit{correlation length} of the solution. \\

	\noindent{\it Keywords:} The stochastic heat equation, intermittency, islands, peaks \\

	\noindent{\it \noindent AMS 2000 subject classification:}
	Primary 60H15; Secondary 35R60.
\end{abstract}

\section{Introduction}
Let $\dot{W}:=\{\dot{W}_t(x)\}_{t>0,x\in\R}$
denote space-time white noise, and
consider the nonlinear stochastic heat equation,
\begin{equation}\label{heat}
	\frac{\partial}{\partial t}u_t(x) = \frac12 \frac{\partial^2}{%
	\partial x^2}u_t(x) +
	\sigma(u_t(x)) \dot W_t(x),
\end{equation}
for $(t\,,x)\in(0\,,\infty)\times\R$,
subject to $u_0(x):=1$ for all $x\in\R$. Throughout we consider only
the case that $\sigma:\R\to\R$ is Lipschitz continuous. In that case,
the theory of Walsh \cite{Walsh} explains the
meaning of \eqref{heat} and shows that \eqref{heat} has a unique 
[strong] solution that is continuous for all $(t\,,x)\in[0\,,\infty)\times\R$.
The goal of this article is to make some observations about the geometric
structure of the random function $x\mapsto u_t(x)$ for $t>0$ fixed. Notice that a consequence of the fact that $u_0(x)$ is constant is that the law of $u_t(x)$ doesn't depend on $x$ (\cite{Dalang:99}). 

Although we will have some results that are valid for \eqref{heat}
in general, we are mainly motivated by the following two special cases 
of Eq.\ \eqref{heat}:
\begin{description}
	\item[Case 1.] There exists $q>0$ such that $\sigma(z)=qz$ for all $z\in\R$.
		In this case, \eqref{heat} is known as the \emph{parabolic Anderson model};
	\item[Case 2.] $0<\inf_{z\in\R}\sigma(z)\le\sup_{z\in\R}\sigma(z)<\infty$.
		An important special case of this case occurs when $\sigma$ is a constant;
		then \eqref{heat} is the linear SPDE whose solution is a 
		stationary \emph{Gaussian} process.
\end{description}
Let $\log_+(x):=\log(x\vee \e)$ and
define, for all $R,\alpha>0$,
\begin{equation}
	g_\alpha(R) := \begin{cases}
		\exp\left( \alpha(\log_+ R)^{\nicefrac23} \right)&\text{in Case 1},\\
		\alpha(\log_+ R)^{\nicefrac12}&\text{in Case 2}.
	\end{cases}
\end{equation}
[``$g$'' stands for ``gauge.'']
Our recent effort \cite{CJK} implies that, for both Cases 1 and
 2, for all $t>0$ fixed there exist
$\alpha_*,\alpha^*>0$ such that $\limsup_{R\to\infty}
u_t(R)/g_\alpha(R)=0$ a.s.\ if $\alpha>\alpha^*$,
and $\limsup_{R\to\infty} u_t(R)/g_\alpha(R)=\infty$
if $\alpha\in(0\,,\alpha_*)$. In other words, the ``exceedence set''
\begin{equation}
	E_\alpha(R) := \left\{ x\in[0\,,R]:\
	u_t(x) \ge g_\alpha(R) \right\}
\end{equation}
is a.s.\ empty for all $R\gg 1$ if $\alpha>\alpha^*$; and
$E_\alpha(R)$ is unbounded for all $R>1$ if
$\alpha\in(0\,,\alpha_*)$. 

Note that the rescaled version $R^{-1}E_\alpha(R)$ of
$E_\alpha(R)$ is a random subset of $[0\,,1]$. One of our original aims 
was to show that $R^{-1}E_\alpha(R)$ ``converges'' to a 
random fractal of Hausdorff dimension $d(\alpha)\in(0\,,1)$
as $R\to\infty$ when $\alpha$ is sufficiently small. 
So far we have not been able to do this, though as we will soon see
we are able to furnish strong evidence in favor of this claim.

If $R^{-1}E_\alpha(R)$ {\it did} look like a 
random fractal subset of $[0\,,1]$ with Hausdorff dimension
$d(\alpha)\in(0\,,1)$, then we would expect its Lebesgue measure to behave as
$R^{-d(\alpha)+o(1)}$ as $R\to\infty$. 
Or stated in more precise terms, we would expect that
if $\alpha$ is sufficiently small, then 
\begin{equation}
	\lim_{R\to\infty} \frac{\log|E_\alpha(R)|}{\log R}= 1-d(\alpha)
	\qquad\text{a.s.}
\end{equation}
The first theorem of this paper comes  close to proving this last assertion.

\begin{theorem}\label{th:sojourn}
	If either Case 1 or Case 2 holds,
	then there exists $\alpha_0>0$ such that for all
	$\alpha\in(0\,,\alpha_0)$ there exists $\gamma\in(0\,,1)$
	such that
	\begin{equation}
		0<\liminf_{R\to\infty}
		\frac{\log \left|E_\alpha(R)\right|}{\log R}\le
		\limsup_{R\to\infty} 
		\frac{\log \left|E_\alpha(R)\right|}{\log R}<1\qquad\text{a.s.}
	\end{equation}
\end{theorem}
The results of \cite{CJK} imply that 
$E_\alpha(R)$ is eventually empty a.s.\ when
$\alpha>\alpha^*$. Therefore, $\alpha_0$ cannot be made to be arbitrarily
large.

Choose and fix a time $t>0$.
Given  two numbers $0<a<b$, we say that a closed interval
$I\subset\R_+$ is an \emph{$(a\,,b)$-island} [at time $t$] if:
\begin{enumerate}
	\item $u_t(\inf I)=u_t(\sup I)=a$;
	\item $u_t(x)>a$ for all $x\in\text{int}(I)$; and
	\item $\sup_{x\in I} u_t(x)>b$.
\end{enumerate}
Define
\begin{equation}
	J_t(a\,,b\,;R) := \text{the length of the largest
	$(a\,,b)$-island $I\subset[0\,,R]$}.
\end{equation}
The following result shows that the relative length of the
largest ``tall island'' in $[0\,,R]$---also known as ``intermittency islands''---is vanishingly small
as $R\to\infty$. This phenomenon has been predicted
[particularly for Case 1] and is a part of a  property
that is referred to somewhat loosely as ``physical intermittency''
\cite{CJK,KPZ,KZ}.

\begin{theorem}\label{th:islands}
	Assume that $\sigma(1)\ne 0$. Then for every $t>0$ and all  $(a\,,b)$ such that
	$1<a<b$ and $\P\{u_t(0)>b\}>0$,
	\begin{equation}
		\limsup_{R\to\infty}\frac{
		J_t(a\,,b\,;R)}{\left| \log R \right|^2}<\infty
		\quad\text{a.s.}
	\end{equation}
	If Case 2 occurs, then the preceding can be improved to
	the following:
	\begin{equation}
		\limsup_{R\to\infty}\frac{J_t(a\,,b\,;R)}{
		\log R\cdot\left|\log \log R\right|^{\nicefrac32}}<\infty
		\quad\text{a.s.}
	\end{equation}
\end{theorem}

Let us make a few remarks before we continue our introduction.
\begin{remark}\begin{enumerate}
	\item During the course of the proof of this theorem,
		we will establish that exists $b>1$
		such that $\P\{u_t(0)>b\}>0$; therefore, the result always
		has content.
	\item The condition $\sigma(1)\neq 0$ is necessary. Indeed, if
		$\sigma(1)$ were zero, then $u_t(x)=1$ for all $t>0$ and
		$x\in\R$ [this is because $u_0\equiv 1$].\qed
\end{enumerate}\end{remark}

Theorems \ref{th:sojourn} and \ref{th:islands} both rely on 
a fairly good estimation of ``correlation length'' for the random
field $x\mapsto u_t(x)$. There are many ways one can understand
the loose term, ``correlation length.'' Let us describe one next.

Let $\{X_x\}_{x\in\R}$ be a random field on $(\Omega\,,\F,\P)$,
and let $\mathcal{L}(\ell)$ denote the collection of all weakly
stationary random fields
$\{Y_x\}_{x\in\R}$ on $(\Omega\,,\F,\P)$ such that
$Y$ has ``lag'' $\ell$; that is, $Y_z$ is independent of
$(Y_{x_i})_{i=1}^N$ for all $z,x_1,\ldots,x_N\in\R$ that
satisfy $\min_{1\le j\le N}|z-x_j|\ge\ell$.
Then, the \emph{correlation length} of $X$ is the function
\begin{equation}\label{def:CL}
	L_X (\epsilon\,;\delta) := \inf\left\{
	\ell>0 :\ \inf_{Y\in\mathcal{L}(\ell)} \sup_{x\in\R}
	\P\{|X_x-Y_x| > \delta\} < \epsilon \right\},
\end{equation}
where $\epsilon,\delta>0$ can be thought of as \emph{fidelity} parameters.
Informally speaking, when we find $L_X(\epsilon\,;\delta)$,
we seek to find the smallest lag-length
$\ell$ for which there exists a coupling of $X$ with a 
lag-$\ell$ process $Y$, such that the coupling is good to within
$\delta$ units with probability at least $1-\epsilon$.

The following is the main technical result of this paper.
It states that the correlation length of the solution to \eqref{heat} is logarithmic
in the fidelity parameter $\epsilon$; and the fidelity parameter $\delta$
can be as small as $\exp(-K\left|\log\epsilon\right|^{%
\nicefrac23})$ for a universal $K=K(t)$.

\begin{theorem}\label{th:main}
	For every $t>0$, there exists a positive
	and finite constant $K:=K(t)$, such that as $\epsilon\downarrow 0$,
	\begin{equation}\label{eq:main1}
		L_{u_t}\left( \epsilon\,; \e^{-K\left|
		\log\epsilon\right|^{\nicefrac23}}
		\right) = O\left( \left| \log\epsilon\right| \right).
	\end{equation}
	If $\sigma$ is a bounded function, then in fact there exists
	$\theta\in(0\,,1)$ such that
	\begin{equation}\label{eq:main2}
		L_{u_t} \left(
		\epsilon\,; \left[\frac{\log\left|\log\epsilon\right|}{%
		\left|\log\epsilon\right|}
		\right]^\theta \right) = O\left( \left[\log\left|\log\epsilon\right|
		\right]^{\nicefrac32}\right)\qquad(\epsilon\downarrow 0).
	\end{equation}
\end{theorem}

Our notion of correlation length is stronger than other, somewhat simpler,
notions of this general type. 
For instance, consider the following: Let $\{X_x\}_{x\in\R}$ be
a random field, and define $L^*_X(\epsilon\,;\delta)$ to be the smallest
$\ell>0$ for which we can find---on \emph{some} probability space---a
coupling $(X^*\,,Y^*)$, where $X^*$ has the same law as $X$ and
$Y^*$ has lag $\ell$, and
$\sup_{x\in\R}\P\{|X_x^*-Y^*_x|>\delta\}<\epsilon$.
Since $L^*_X(\epsilon\,;\delta)\le
L_X(\epsilon\,;\delta)$, Theorem \ref{th:main} readily implies
that
\begin{equation}
	L^*_{u_t} \left(
	\epsilon\,; \e^{-K\left|
	\log\epsilon\right|^{\nicefrac23}} \right)
	= O(\left|\log\epsilon\right|)
	\hskip1in (\epsilon\downarrow 0).
\end{equation}

\noindent\textbf{Open Problem.} Is it true that
$L^*_{u_t}(\epsilon\,;0)=O(\left|\log\epsilon\right|)$?
This is equivalent to asking whether or
not $x\mapsto u_t(x)$ is exponentially mixing.\\

Although we do not know how to prove that $x\mapsto u_t(x)$
is exponentially mixing, we are able to prove that the coupling
in Theorem \ref{th:main} is ``good on all scales.'' In order to interpret this,
note that if $\ell:=L_{u_t}(\epsilon;\,\delta)$ then we can 
basically approximate
$u_t$ well enough by a random field $Y$ in $\mathcal{L}(\ell)$ such that
$Y$ replicates $u_t$ to within $\delta$ units. According to 
\eqref{eq:main1} this can be done with ---$\ell=O(\left|\log\epsilon\right|)$---with $\delta$ having the form $\exp\{-K\left|\log\epsilon\right|^{\nicefrac23}\}$
for some $K:=K(t)$. Thus, for example, if we wanted to know how small 
$x\mapsto u_t(x)$ can possibly get, then we could study instead
$Y$ provided that ``how small'' means ``$\exp\{-K\left|\log\epsilon\right|^{\nicefrac23}\}$
or more.'' Our next result shows that this notion of ``how small'' is generic
[and not at all a restriction]. Our proof borrows several important ideas
from a paper by Mueller and Nualart \cite{MN}.

\begin{theorem}\label{th:MN}
	If $\sigma(0)=0$, then for every $t,a>0$ and $x\in\R$,
	\begin{equation}
		\lim_{\epsilon\downarrow 0}
		\frac{1}{\left|\log\epsilon\right|}
		\log\P\left\{ u_t(x) \le \e^{-a \left|\log\epsilon\right|^{\nicefrac23}}
		\right\} =-\infty.
	\end{equation}
\end{theorem}

Throughout this paper, ``log'' denotes the natural logarithm,
$p_t(x)$ denotes the standard heat kernel for $(\nicefrac12)\Delta$,
\begin{equation}
	p_t(x) := \frac{\e^{-x^2/(2t)}}{(2\pi t)^{\nicefrac12}}
	\qquad(t>0~,~x\in\R),
\end{equation}
and $\|Z\|_k:=\{\E(|Z|^k)\}^{1/k}$ denotes the $L^k(\P)$-norm of
a random variable $Z\in L^k(\P)$ $(k\in[1\,,\infty))$.

Let us conclude the Introduction with a brief outline of the paper.
In Section \ref{sec:main} we prove Theorem \ref{th:main},
whose corollaries, Theorems \ref{th:sojourn} and \ref{th:islands},
are proved respectively in \S\ref{sec:sojourn} and \S\ref{sec:islands}.
In a final Section \ref{sec:NegMoments} we state and prove an improved version
of Theorem \ref{th:MN},
which might turn out to be a first step in answering the mentioned
Open Problem.

\section{Proof of Theorem \ref{th:main}}\label{sec:main}
First of all, recall that the solution to the stochastic PDE \eqref{heat}
is the unique continuous solution to the following random
evolution equation \cite{Walsh}:
\begin{equation}\label{mild}
	u_t(x) = 1 + \mathop{\int}_{(0,t)\times\R}
	p_{t-s}(y-x) \sigma\left(
	u_s(y)\right)\, W(\d s \,\d y).
\end{equation}
For all $\beta>0$, let $U^{(\beta)}$ solve the following
closely-related stochastic evolution equation:
\begin{equation}
	U^{(\beta)}_t(x) = 1 +\hskip-1cm
	\mathop{\int}_{(0,t)\times[x-\sqrt{\beta t},
	x+\sqrt{\beta t}]}\hskip-1cm
	p_{t-s}(y-x) \sigma\left(
	U^{(\beta)}_s(y)\right)\, W(\d s \,\d y).
\end{equation}
It has been observed in \cite{CJK}
that the same methods as in \cite{Walsh} can be used
to show that there exists a unique continuous
random field $U^{(\beta)}$ that solves the preceding.
The following result of \cite{CJK} shows that 
$U^{(\beta)}\approx u$ if $\beta$ is large.

\begin{lemma}[\protect{\cite[Lemma 4.2]{CJK}}]\label{lem:u-Ubeta}
	For every $T>0$ there exists finite and positive constants 
	$a_i$ $[i=1,2]$ such that for all $\beta>0$,
	and for all real numbers $k\in[1\,,\infty)$,
	\begin{equation}
		\sup_{\substack{t\in(0,T)\\x\in\R}} \E\left( \left|
		u_t(x)-U^{(\beta)}_t(x)\right|^k\right) \le
		a_1^k \e^{ a_1k\left[ k^2-a_2\beta\right]}.
	\end{equation}
\end{lemma}

It is easy to adapt the arguments of \cite{CJK} to improve the preceding
in the case that $\sigma$ is bounded. Because all of the key steps
are already in Ref.\ \cite{CJK}, we state the end result without proof.

\begin{lemma}\label{lem:u-Ubeta:sigma:bdd}
	Suppose, in addition, that $\sigma$ is bounded.
	Then for every $T>0$ there exists finite and positive constants 
	$\bar{a}_i$ $[i=1,2]$ 
	such that for all $\beta>0$,
	and for all real numbers $k\in[1\,,\infty)$,
	\begin{equation}
		\sup_{\substack{t\in(0,T)\\x\in\R}} \E\left( \left|
		u_t(x)-U^{(\beta)}_t(x)\right|^k\right) \le
		\bar a_1^k \e^{ \bar a_1k\left[ \log k-\bar a_2\beta\right]}.
	\end{equation}
\end{lemma}

The process $U^{(\beta)}$ is useful only as a first step in a better
coupling, which we describe next. Define
$U^{(\beta,\,0)}_t(x) := 1$. Then, once $U^{(\beta,l)}$ is defined
[for some $l\ge 0$] we define $U^{(\beta,l+1)}$ as follows:
\begin{equation}
	U^{(\beta,\,l+1)}_t(x) := 1 +\hskip-1cm \mathop{\int}_{(0,t)\times[x-\sqrt{\beta t},
	x+\sqrt{\beta t}]} \hskip-1cm
	p_{t-s}(y-x) \sigma\left(
	U^{(\beta,\,l)}_s(y)\right)\, W(\d s \,\d y).
\end{equation}
In other words, $U^{(\beta,l)}$ is the $l^{\mbox{\scriptsize th}}$ step in the Picard-iteration
approximation to $U^{(\beta)}$. The following result of
\cite{CJK} tells us that if $l$ is large then $U^{(\beta,l)}\approx U^{(\beta)}$.

\begin{lemma}[\protect{\cite[Eq.\ (4.22) \& Lemma 4.4]{CJK}}]%
	\label{lem:Ubeta-Ubetal}
	For every $T>0$ there exists finite and positive constants 
	$b_i$ $[i=1,2]$ such that for all $\beta>0$,
	all integers $n\ge 0$,
	and for all real numbers $k\in[1\,,\infty)$,
	\begin{equation}
		\sup_{\substack{t\in(0,T)\\x\in\R}} \E\left( \left|
		U^{(\beta)}_t(x)-U^{(\beta,n)}_t(x)\right|^k\right) \le
		b_1^k \e^{ b_1k\left[ k^2-b_2n\right]}.
	\end{equation}
	Furthermore, $U^{(\beta,n)}_t \in \mathcal{L}(2n\sqrt{\beta t})$
	for all $\beta,t>0$ and $n\ge 0$.
\end{lemma}
Once again, we state---without proof---an improvement in the case
that $\sigma$ is bounded.

\begin{lemma}\label{lem:Ubeta-Ubetal:sigma:bdd}
	Suppose, in addition, that $\sigma$ is bounded.
	Then, for every $T>0$ there exists finite and positive constants 
	$\bar b_i$ $[i=1,2]$ 
	such that for all $\beta>0$,
	all integers $n\ge 0$,
	and for all real numbers $k\in[1\,,\infty)$,
	\begin{equation}
		\sup_{\substack{t\in(0,T)\\x\in\R}} \E\left( \left|
		U^{(\beta)}_t(x)-U^{(\beta,n)}_t(x)\right|^k\right) \le
		\bar b_1^k \e^{\bar b_1k\left[ \log k-\bar b_2n\right]}.
	\end{equation}
\end{lemma}

Now we are ready to establish Theorem \ref{th:main}.

\begin{proof}[Proof of Theorem \ref{th:main}]
	Choose and fix $t>0$.
	The final assertion of Lemma \ref{lem:Ubeta-Ubetal} implies
	that the process $x\mapsto
	Y_x:= U^{(\beta,n)}_t(x)$ is in $\mathcal{L}(2n\sqrt{\beta t})$
	for every $\beta>0$ and $n\ge 0$. Therefore, we may apply
	Lemmas \ref{lem:u-Ubeta} and \ref{lem:Ubeta-Ubetal} in conjunction
	with Chebyshev's inequality to see that for all $k\in[1\,,\infty)$
	and $\delta>0$,
	\begin{equation}\label{eq:basic0}
		\inf_{Y\in\mathcal{L}(2n\sqrt{\beta t})}
		\sup_{x\in\R}\P\left\{ |u_t(x)-Y_x|>\delta\right\}
		\le (2c_1/\delta)^k \e^{c_1k[k^2-c_2(\beta\wedge n)]},
	\end{equation}
	where $c_1:=\max\{a_1\,,b_1\}$, $c_2:=\min\{(a_1a_2)/c_1,\,(b_1b_2)/c_1\}$ do not depend on $(\beta\,,n\,,k\,,\delta)$. Now we choose 
	$\beta = n := 1+\lfloor (2/c_2) k^2\rfloor$ in order to find that
	there exists $\bar c\in(1\,,\infty)$ such that for all $k$ sufficiently large,
	\begin{equation}\label{eq:basic1}
		\inf_{Y\in\mathcal{L}(\bar c k^3)}
		\sup_{x\in\R}\P\left\{ |u_t(x)-Y_x|>\delta\right\}
		\le \delta^{-k}\e^{-2k^3/\bar c}.
	\end{equation}
	Because $\bar c$ does not depend on $\delta$, we can set $\delta:=
	\exp(- k^2/\bar c)$ to deduce from the preceding
	that for every $\nu\in(0\,,1)$ fixed,
	\begin{equation}
		L_{u_t} \left( \e^{-k^3/\bar c}\,;\e^{- k^2/\bar c}
		\right) \le \bar c k^3,
	\end{equation}
	uniformly for all $k$ sufficiently large. It follows that if
	$\epsilon:=\exp(-k^3/\bar c)$, then
	\begin{equation}
		L_{u_t} \left( \epsilon\,;\exp\left\{-\frac{%
		\left|\log\epsilon\right|^{\nicefrac23}}{\bar c^{\nicefrac13}}
		\right\}
		\right) \le \bar{c}^2\left|\log\epsilon\right|.
	\end{equation}
	In the case that $\epsilon$ is a general positive number, 
	\eqref{eq:main1} follows from the preceding and a
	simple monotonicity argument.
	
	In the case that $\sigma$ is bounded, we proceed similarly as in the general case,
	but apply Lemmas \ref{lem:u-Ubeta:sigma:bdd} and 
	\ref{lem:Ubeta-Ubetal:sigma:bdd} in
	place of Lemmas \ref{lem:u-Ubeta} and \ref{lem:Ubeta-Ubetal}, and then
	select the various parameters accordingly. In this way, we find 
	the following improvement to \eqref{eq:basic0} in the case that
	$\sigma$ is bounded:
	\begin{equation}\label{eq:basic2}
		\inf_{Y\in\mathcal{L}(2n\sqrt{\beta t})}
		\sup_{x\in\R}\P\left\{ |u_t(x)-Y_x|>\delta\right\}
		\le (2c_1'/\delta)^k \e^{c_1'k[\log k-c_2'(\beta\wedge n)]},
	\end{equation}
	where $c_1',\,c_2'$ do not depend on $(\beta\,,n\,,k\,,\delta)$.
	Now we choose $\beta=n:=1+\lfloor (2/c_2')\log k\rfloor$
	in order to deduce the existence of a constant $c''\in(1\,,\infty)$
	such that for all suffiently large $k$,
	\begin{equation}\label{eq:basic3}
		\inf_{Y\in\mathcal{L}\left( c'' [\log k]^{\nicefrac32} \right)}
		\sup_{x\in\R}\P\left\{ |u_t(x)-Y_x|>\delta\right\}
		\le \delta^{-k}\e^{-2k\log k/ c''}.
	\end{equation}
	This is our improvement to \eqref{eq:basic1} in the case that
	$\sigma$ is bounded. In particular,
	for all $k$ large,
	\begin{equation}\label{eq:basic4}
		\inf_{Y\in\mathcal{L}(c''[\log k]^{\nicefrac32})}
		\sup_{x\in\R}\P\left\{ |u_t(x)-Y_x|> k^{-1/c''}\right\}
		\le \e^{-k\log k/c''}.
	\end{equation}
	If $\epsilon:=\exp\{-(1/c'')k\log k\}$ is small, then
	$k\approx c''\left|\log\epsilon\right|/\log\left|\log\epsilon\right|$
	and \eqref{eq:main2} follows from \eqref{eq:basic4}
	for every $\theta\in(0\,,1/c'')$.
	We apply monotonicity in order to deduce \eqref{eq:main2}
	for general [small] $\epsilon$.
\end{proof}

\section{Proof of Theorem \ref{th:sojourn}}\label{sec:sojourn}
Before we proceed with the proof we need a few technical results.
Suppose $Y\in\mathcal{L}(\ell)$ for some $\ell>0$, and define,
for all integers $n\ge 1$ and real numbers $\alpha>0$,
\begin{equation}
	\mathfrak{Y}_\alpha(n) := \int_0^{n\ell} \1_{\{ Y_x\ge
	\bar G((n\ell)^{-\alpha})\}}\,\d x,
\end{equation}
where
\begin{equation}\label{G}
	\bar G(a) := \sup\left\{ b>0:\, \P\{ Y_0\ge b \}\ge a\right\}.
\end{equation}

\begin{lemma}\label{lem:Y} 
	Assume that $Y\in\mathcal{L}(\ell)$ for some $\ell\ge 1$. 
	 Then for every
	integer $k\ge 3$ there exists a universal constant $C_k\in(0\,,\infty)$
	such that for all $\alpha\in\left(0\,,\nicefrac12 \right)$ and  $n\ge 2$,
	\begin{equation}
		\left\| \frac{\mathfrak{Y}_\alpha(n)}{\E\mathfrak{Y}_\alpha(n) }-1
		\right\|_k \le C_k\cdot \frac{\ell^{\alpha}}{n^{\frac{1}{2}-\alpha}}.
	\end{equation}
\end{lemma}

\begin{proof}
	We can write
	\begin{equation}
		\mathfrak{Y}_\alpha(n) := \sum_{j=0}^{n-1} Z_j,
		\quad\text{where}\quad
		Z_j := \int_{j\ell}^{(j+1)\ell} \1_{\{ Y_x\ge
		\bar G((n\ell)^{-\alpha})\}}\,\d x.
	\end{equation}
	
	Define
	\begin{equation}
		S_n^{\textnormal{(o)}} := \sum_{0\le 2j+1\le n-1} \left( Z_{2j+1}
		-\E Z_{2j+1} \right),\quad
		S_n^{\textnormal{(e)}} := \sum_{0\le 2j\le  n-1} \left( Z_{2j}-
		\E Z_{2j}\right).
	\end{equation}
	It follows that 
	\begin{equation}\label{S_n:decomp}
		\mathfrak{Y}_\alpha(n)-\E\mathfrak{Y}_\alpha(n)= 
		S^{(\textnormal{o})}_n + S^{(\textnormal{e})}_n.
	\end{equation}
	The processes $S^{(\textnormal{o})}$ and $S^{(\textnormal{e})}$
	are mean-zero random walks, and hence martingales [in their respective filtrations].
	Define
	\begin{equation}
		X^{(\textnormal{x})}_k := S^{(\textnormal{x})}_k-
		S^{(\textnormal{x})}_{k-1}\qquad
		(k\ge 1)
	\end{equation}
	to be the increments of $S^{(\textnormal{x})}$ for $\textnormal{x}\in
	\{\textnormal{o}\,,\textnormal{e}\}$, and
	$\mathcal{G}_k^{(\textnormal{x})}$ the sigma-algebra generated by 
	$\{X^{(\textnormal{x})}_j\}_{j=1}^k$.
	Because $\textnormal{Var}(Z_1)\le \E(Z_1^2)\le \ell\E(Z_1)\le
	\ell^{2} $, $|X^{(\textnormal{x})}_j| < \ell$ for every $j$, 
	and since $\ell\ge 1$, an application of Burkholder's inequality 
	\cite{Burkholder} (specifically, see Hall and Heyde
	\cite[Theorem 2.10, p.\ 23]{HallHeyde}) implies that for every $k\ge 1$
	there exists a universal constant $c_k\in(0\,,1)$ such that
	for every $k\ge 2$ and $n\ge 2$
	and $\textnormal{x}\in\{\textnormal{o}\,,\textnormal{e}\}$,
	\begin{equation}
		c_k^k \E\left(  \big\vert S^{(\textnormal{x})}_n  
		\big\vert^k\right)
		\le n^{k/2}\ell^k + \ell^k\le
		2n^{k/2}\ell^k .
	\end{equation}
	The lemma follows from the above,  \eqref{S_n:decomp}, and 
	Minkowski's inequality 
	together with the observation that 
	$\E\mathfrak{Y}_\alpha(n) \ge (n\ell)^{1-\alpha}$.
\end{proof}

\begin{proof}[Proof of Theorem \ref{th:sojourn}]
	Throughout the demonstration, we choose and fix a time
	$t>0$. 
	
	Consider first Case 1. According to \cite{CJK}, 
	there exist constants $A_1,\ldots,A_4\in(0\,,\infty)$
	such that for all $\lambda\ge 1$ and $x\in\R$,
	\begin{equation}\label{eq:tail:PAM}
		A_1 \e^{-A_2(\log\lambda)^{\nicefrac32}}\le
		\P\left\{ u_t(x) > \lambda\right\} \le 
		A_3 \e^{-A_4(\log \lambda)^{\nicefrac32}}.
	\end{equation}
	[One could prove that the preceding probability does not depend on $x\in\R$.]
	
	According to Theorem \ref{th:main}, for every $m> 1$ there exists
	$c\in(0\,,\infty)$ such that for all 
	$R$ large enough,
	we can find a process $Y\in\mathcal{L}(c\log R)$ such that
	\begin{equation}
		\P\left\{ |u_t(x)-Y_x| \ge 1\right\} \le \text{const}\cdot R^{-m}.
	\end{equation}
	[One could prove that the preceding probability does not depend on $x\in\R$.]
	Note, in particular, that
	\begin{equation}\label{u-Y}\begin{split}
		\P\left\{ \int_0^R\1_{\{|u_t(x)-Y_x|\ge 1\}}\,\d x
			\ge 1 \right\} &\le \E\left(\int_0^R
			\1_{\{|u_t(x)-Y_x|\ge 1\}}\,\d x\right)\\
		& \le \text{const}\cdot R^{1-m}.
	\end{split}\end{equation}
	
	For all $\alpha\in(0\,,1)$ and $R$ large enough,
	\begin{equation}\begin{split} \label{eq:tail:coupling}
		\P\left\{ Y_x \ge\e^{\alpha(\log R)^{\nicefrac23}} \right\} &\le 
			\P\left\{ u_t(x)\ge \e^{(\alpha/2)(\log R)^{\nicefrac23}}\right\}
			+\textnormal{const}\cdot R^{-m}\\
		&\le A_3 R^{-A_4(\alpha/2)^{\nicefrac32}} + \textnormal{const}
			\cdot R^{-m}\\
		&\le\textnormal{const}\cdot  R^{-A_4(\alpha/2)^{\nicefrac32}},
	\end{split}\end{equation}
	provided that $m>A_4$. Similarly,
	\begin{equation}\begin{split}
		\P\left\{ Y_x\ge \e^{\alpha(\log R)^{\nicefrac23}}\right\}
			&\ge A_1 R^{-A_2(2\alpha)^{\nicefrac32}} - \textnormal{const}\cdot
			R^{-m} \\
		&\ge \textnormal{const}\cdot R^{-A_2(2\alpha)^{\nicefrac32}}.
	\end{split}\end{equation}
	provided that
	$m>(2A_2)^{\nicefrac32}$. We combine the preceding two bounds, and 
	then relabel $\alpha$ to see that
	there exist $B_1,B_2,B_3,B_4\in(0\,,\infty)$
	and $\alpha_0\in(0\,,\nicefrac14)$ such that for all $\alpha\in(0\,,\alpha_0)$,
	\begin{equation}\label{1}
		B_1 \e^{B_2(\alpha \log R)^{\nicefrac23}}
		\le -1+\bar G (R^{-\alpha}) \le 
		1+ \bar G(R^{-\alpha}) \le B_3 \e^{B_4(\alpha \log R)^{\nicefrac23}},
	\end{equation}
	where $\bar{G}$ was defined in \eqref{G}. According to \eqref{u-Y},
	\begin{equation}\label{2}\begin{split}
		&\P\left\{ \int_0^R \1_{\{ u_t(x) \ge 1+\bar G(R^{-\alpha})\}}\,\d x
			\ge 1+ \int_0^R\1_{\{Y_x\ge \bar{G}(R^{-\alpha})\}}\,\d x \right\}\\
		&\hskip3in\leq \textnormal{const}\cdot R^{1-m},
	\end{split}\end{equation}
	and
	\begin{equation}\label{3}\begin{split}
		&\P\left\{ \int_0^R \1_{\{ u_t(x) \ge -1+\bar G(R^{-\alpha})\}}\,\d x
			\le -1+ \int_0^R\1_{\{Y_x\ge \bar{G}(R^{-\alpha})\}}\,\d x \right\}
			\\
		&\hskip3in\leq \textnormal{const}\cdot R^{1-m}.
	\end{split}\end{equation}
	We emphasize
	that``const'' does not depend on $R$ in  the previous two displays.
	We may apply Lemma \ref{lem:Y} with $\ell:=c\log R$,
	$n:=R/\ell$ and $\mathfrak{Y}_{\alpha}(n)=\int_0^R \1_{\{ Y_x\ge
	\bar G(R^{-\alpha})\}}\,\d x$ to see that for all $k\ge 2,\,R$ sufficiently large,
	\begin{equation}
		\P\left\{ \left|
		\int_0^R\frac{\1_{\{Y_x\ge \bar{G}(R^{-\alpha})\}}}{%
		\E\mathfrak{Y}_{\alpha}(n)}
		\,\d x  - 1\right|
		\ge R^{-\alpha}\right\} \le \textnormal{const}\cdot 
		\frac{(\log R)^{k/2}}{R^{ k(1-4\alpha)/2}}.
	\end{equation}
	Let us pause and recall that $\alpha<\alpha_0<\nicefrac14$, so that the
	right-hand side is at most $R^{-2}$ provided that we have chosen $k$
	sufficiently large. We also note that there exists $\gamma\in(0\,,1)$
	such that
	\begin{equation}
	 	R^{1-\alpha} \le  \E \mathfrak{Y}_{\alpha} (n) \le R^{1-\alpha \gamma},
	\end{equation}
	for all sufficiently large $R$; see
	\eqref{eq:tail:PAM} and \eqref{eq:tail:coupling}.
	Therefore, we can combine \eqref{1}, \eqref{2},
	and \eqref{3}, together with the Borel--Cantelli lemma to see that 
	as long as $\alpha_0$ were selected sufficiently small,
	\begin{equation}
		0<\liminf_{\substack{R\to\infty\\ R\in\Z}}
		\frac{\log \left|E_\alpha(R)\right|}{\log R}\le
		\limsup_{\substack{R\to\infty\\ R\in\Z}} 
		\frac{\log \left|E_\alpha(R)\right|}{\log R}<1\qquad\text{a.s.}
	\end{equation}
	A monotonicity argument finishes the proof for Case 1.
	
	Case 2 is proved similarly, but we apply the following estimate \cite{CJK}
	in place of \eqref{eq:tail:PAM}:
	$C_1 \exp(-C_2 \lambda^2)\le
	\P\{ u_t(x) > \lambda\} \le C_3 \exp(-C_4 \lambda^2)$
	$(\text{for }\lambda\ge 1)$. We omit the details.
\end{proof}

\section{Proof of Theorem \ref{th:islands}}\label{sec:islands}
	First of all, let us note that the conditions of the theorem are non vacuous.
	In other words, we need to prove that there exists $b>1$
	such that $\P\{u_t(0)\ge b\}>0$. Because
	$\E u_t(0)=1$, it follows that there exists $b\ge 1$ such that
	$\P\{u_t(0)\ge b\}>0$. Suppose to
	the contrary that $\P\{u_t(0)>1\}=0$. Then, $u_t(0)$ is a.s.\
	equal to 1. It follows that
	that the stochastic integral in \eqref{mild}
	vanishes a.s.\ for  $x=0$. The corresponding quadratic variation must too;
	that is,
	\begin{equation}
		\int_0^t \d s\int_{-\infty}^\infty 
		\d y\ \left[ p_{t-s}(y) \sigma\left(u_s(y)\right)\right]^2
		=0\quad\text{a.s.}
	\end{equation}
	Since the heat kernel never vanishes, we find that $\sigma\left(u_s(y)\right)=0$
	for almost all $(s\,,y)\in(0\,,t)\times \R$, whence for
	all $(s\,,y)\in(0\,,t)\times\R$ by continuity. This is a contradiction since $u_0\equiv 1$. Therefore,
	there exists $b>1$ such that $\P\{u_t(0)\ge b\}>0$. 
	Now we proceed with our proof of
	the bulk of Theorem \ref{th:islands}.
	
	Theorem \ref{th:islands} is a simple consequence of Theorem \ref{th:main}
	together with ideas that are borrowed from a classical paper by
	Erd\H{o}s and R\'enyi \cite{ErdosRenyi} on the length of the longest
	run of heads in an infinitely-long sequence of independent coin tosses. 
	
	Choose and fix two integer $R,m\gg 1$ and a  real $\delta\in(0\,,1)$
	small enough that $a-2\delta > 1$ and $\P\{u_t(0)>b + 2\delta\}>0$. 
	According to Theorem \ref{th:main} we can
	find a constant $c\in(0\,,\infty)$---independent of $R$---and a random
	field $Y\in\mathcal{L}(c\log R)$ such that
	\begin{equation}
		\P\left\{ |u_t(x)-Y_x| > \delta\right\} \le \frac{c}{R^m}.
	\end{equation}
	[One can prove that the probability does not depend on $x$.] 
	Define $x_j:= cj\log R$ for
	all non negative integers $j$, and observe that
	\begin{equation}\label{eq:u-Y}
		\P\left\{ \max_{0\le j\le  \lfloor R/(c\log R)\rfloor}
		\left| u_t(x_j)-Y_{x_j}\right| >\delta\right\}\le \textnormal{const}
		\cdot R^{1-m}.
	\end{equation}
	
	Let us call the index $j$ ``good'' if
	$Y_{x_j},Y_{x_{j+2}}<a-\delta$ and $Y_{x_{j+1}}>b+\delta$. Otherwise
	$j$ is deemed ``bad.''
	Clearly,
	\begin{equation}\begin{split}
		p&:=\P\left\{ j\text{ is good} \right\} \\
		&=\left( \P\left\{ Y_0<a-\delta\right\}\right)^2\cdot
			\P\left\{ Y_0>b+\delta\right\}\\
		&\ge \left( \P\{ u_t(0)<a-2\delta\} - \frac{c}{R^m}\right)^2\cdot
			\left( \P\{u_t(0)>b+2\delta\} - \frac{c}{R^m}\right).
	\end{split}\end{equation}
	We may observe that $p$ does not depend on $j$. Moreover,
	$\P\{u_t(0)<a-2\delta\}\wedge \P\{u_t(0)>b+2\delta\}>0$
	because of the choice of $(b\,,\delta)$
	and the fact that $\E u_t(0)=1<a-2\delta$.
	Therefore, we may choose $m$ large enough to ensure that $p>0$;
	note that we may also choose $m$ independently of $R\gg 1$. 
	
	Because 
	\begin{equation}
		\P\left\{  j\,,j+3\,,\ldots,j+3n\text{ are all bad}\right\}
		=(1-p)^n,
	\end{equation}
	it follows that
	\begin{equation}\begin{split}
		&\P\left\{{}^\exists \, 0\le j\le
			\left\lfloor \frac{R}{c\log R}\right\rfloor:\,
			 j\,,j+3\,,\ldots,j+3\lfloor \gamma\log R\rfloor
			\text{ are all bad}\right\}\\
		&\hskip3.3in\le \textnormal{const} \cdot R^{-2},
	\end{split}\end{equation}
	provided that $\gamma$ is a sufficiently-large universal constant.
	This and the Borel--Cantelli lemma together imply
	that a.s.\  for all sufficiently-large integers $R$, the maximum 
	distance between two good points is at most
	$6\gamma \log R\cdot c\log R$. Combined with \eqref{eq:u-Y}, 
	we can conclude that the size of the largest 
	island is at most $6c\gamma (\log R)^2$. This proves the
	theorem for Case 1. 
	
	If Case 2 holds, then we proceed
	exactly as we did above, but can find our random field
	$Y\in \mathcal{L}(c[\log\log R]^{\nicefrac32})$
	instead of $\mathcal{L}(c\log R)$. The remaining details are
	omitted.
\qed

\section{Proof of Theorem \ref{th:MN}}\label{sec:NegMoments}
We conclude by proving Theorem \ref{th:MN}. Throughout
we assume that
\begin{equation}
	\sigma(0)=0.
\end{equation}
[This of course includes Case 1.]
In that case Mueller's comparison principle  \cite{CJK,Mueller} guarantees that
$u_t(x)\ge 0$ for all $t>0$ and $x\in\R$ a.s.
We offer the following quantitative improvement, which clearly implies Theorem \ref{th:MN}:

\begin{theorem}\label{th:NegMoments}
	For every $t>0$ there exist
	$A,B\in(0\,,\infty)$  such that uniformly for all
	$\epsilon\in(0\,,1)$ and $x\in\R$,
	\begin{equation}
		\P \{ u_t(x) < \epsilon \}\le A\exp\left( -B\left\{ \left| 
		\log\epsilon\right|\cdot\log\left|\log\epsilon\right|
		\right\}^{\nicefrac32} \right).
	\end{equation}
\end{theorem}
Before we prove this result, let us state and prove two corollories
to Theorem \ref{th:NegMoments}. The corollaries are of independent
interest, but also showcase the usefulness of quantitative estimates 
in this area. The first corollary identifies an upper bound for the 
exponential growth of the high negative moments of $u_t(x)$ when
$\sigma(0)=0$. We believe that the rate provided below is sharp.

\begin{corollary}\label{co1}
	For all $t>0$ and $x\in\R$,
	\begin{equation}
		\limsup_{k\to\infty} \left[
		\left( \frac{\log k}{k}\right)^3 
		\log \E \left( | u_t(x)|^{-k} \right)\right]<\infty.
	\end{equation}
\end{corollary}

\begin{proof}
	Theorem \ref{th:NegMoments} implies that 
	$u_t(x)>0$ a.s., whence $X:=1/u_t(x)$ is well defined. Because
	 $\E(X^k)=k\int_0^\infty
	\lambda^{k-1}\P\{X>\lambda\}\,\d\lambda$, we can divide the
	integral into
	two pieces where: (i) $\lambda<\e$; and (ii) $\lambda\ge\e$. In
	this way we find that
	\begin{equation}
		\E(X^k) 
		\le \e^k+Ak\cdot\int_1^\infty \e^{f_k(s)}\,\d s,
	\end{equation}
	where
	\begin{equation}
		f_k(s):=ks-B(s\log s)^{\nicefrac32}.
	\end{equation} 
	Laplace's method [and/or the method of stationary
	phase] tells us that
	\begin{equation}
		\log\E(X^k)\le(1+o(1))\sup_{s\ge 1} f_k(s)
		\qquad\text{as $k\to\infty$}.
	\end{equation}
	The corollary follows from this and a series of elementary
	estimates which we omit.
\end{proof}
We mention [and verify] the second corollary to Theorem 
\ref{th:NegMoments} next. This corollary describes a
bound for how close
$u_t(x)$ can come to zero, as $x\to\infty$.

\begin{corollary}\label{co2}
	For all $t>0$ and all $\zeta>\zeta_0$ for some $\zeta_0>0$,
	\begin{equation}
		\lim_{x\to\infty}\left[ \e^{\zeta(\log x)^{\nicefrac23}}
		u_t(x)\right]=\infty\qquad\text{a.s.}
	\end{equation}
\end{corollary}

\begin{proof}
	Let $\gamma>0$ be fixed, and define, for every $n\ge 1$,
	a set $A(n)$ as the following finite collection of points
	in the interval $[n\,,2n]$:
	\begin{equation}
		A(n) := \left\{ n + jn^{-\gamma} \right\}_{j=0}^{%
		1+\lfloor n^{1-\gamma}\rfloor}.
	\end{equation}
	According to Theorem \ref{th:NegMoments},
	for all $\zeta>0$ large enough,
	\begin{equation}\label{eq:lower}
		\P\left\{ \inf_{x\in A(n)} u_t(x)< 
		3\e^{-\zeta(\log n)^{\nicefrac23}}\right\}
		=O\left( n^{-2} \right)
		\qquad\text{as }n\to\infty.
	\end{equation}
	
	According to \cite[Lemma A.3]{FK}, there exists an $c\in(0\,,\infty)$
	such that for all $t>0$, $k\in[2\,,\infty)$, and $x,y\in\R$,
	\begin{equation}\label{eq:FK:space}
		\E\left( | u_t(x)-u_t(y)|^k\right) \le
		\e^{ck^3t}|x-y|^{k/2}.
	\end{equation}
	Therefore, a careful appeal to the Kolmogorov continuity theorem
	implies that for all $k\in[2\,,\infty)$, $\eta\in(0\,,1)$, and $t>0$ 
	\begin{equation}
		B_k:= B_k(t\,,\eta):=\sup_I
		\E\left( \sup_{\substack{x,y\in I\\x\neq y}} 
		\frac{|u_t(x)-u_t(y)|^k}{
		|x-y|^{k\eta/2}}\right) <\infty,
	\end{equation}
	where ``$\sup_I$'' denotes the supremum over all 
	closed intervals
	$I\subset\R$ of length one. We omit the details, as they are standard.
	
	Next we apply Chebyshev's inequality to see that for
	every $t>0$, $\eta\in(0\,,1)$, and $k\in[2\,,\infty)$,
	\begin{equation}\begin{split}
		&\P\left\{ \sup_{\substack{n\le x,y\le 2n\\
			|x-y|\le n^{-\gamma}}} |u_t(x)-u_t(y)| \ge
			2\e^{-\zeta(\log n)^{\nicefrac23}}
			\right\} \\
		&\hskip1in\le \sum_{j=0}^{n-1}
			\P\left\{ \sup_{\substack{j\le x,y\le j+1\\
			|x-y|\le n^{-\gamma}}} |u_t(x)-u_t(y)| \ge
			\e^{-\zeta(\log n)^{\nicefrac23}}
			\right\}\\
		&\hskip1in
			\le B_k n^{1-(k\gamma\eta /2)}\e^{\zeta k(\log n)^{\nicefrac23}}
			=O\left( n^{1-(k\gamma\eta/2)+o(1)}\right).
		\label{eq:lower2}
	\end{split}\end{equation}
	Now we choose and fix $k>4/(\eta\gamma)$ so that the left-hand side of
	\eqref{eq:lower2} sums [in $n$]. It follows from \eqref{eq:lower},
	\eqref{eq:lower2}, and the triangle inequality that, for every $\zeta>0$ large enough,
	\begin{equation}
		\sum_{n=1}^\infty\P\left\{ \inf_{x\in(n,2n)}
		u_t(x) < \e^{-\zeta(\log n)^{\nicefrac23}}\right\}<\infty.
	\end{equation}
	The Borel--Cantelli lemma completes the proof.
\end{proof}

\begin{proof}[Proof of Theorem \ref{th:NegMoments}]
	We are going to prove that for all $n\ge 1$,
	\begin{equation}
		\P\left\{ \inf_{x\in(-1,1)}u_t(x) \le \e^{-n}\right\}
		\le A\exp\left( -B (n\log n)^{\nicefrac32} \right).
	\end{equation}
	Since the distribution of $u_t(x)$ does not depend on $x$,
	this is a stronger result than the one advertised by the statement
	of the theorem.
	
	Let $v_t(x)$ denote the unique continuous
	solution to \eqref{heat} subject to
	$v_0(x)=\mathbf{1}_{(-1,1)}(x)$.
	Because $v_0(x)\le 1=u_0(x)$ for all $x$, 
	Mueller's comparison principle \cite{Mueller} tells us that
	there exists a null set off which $u_t(x)\ge v_t(x)$. Therefore, it
	suffices to prove that for all $n\ge 1$,
	\begin{equation}\label{goal:v}
		\P\left\{ \inf_{x\in(-1,1)}v_t(x) \le \e^{-n}\right\}
		\le A\exp\left( -B (n\log n)^{\nicefrac32} \right).
	\end{equation}
	Set $T_0:=0$, and then define iteratively
	\begin{equation}
		T_{k+1} := \inf\left\{ 
		s> T_k:\,  \inf_{x\in(-1,1)}v_s(x) \le \e^{-k-1}
		\right\},
	\end{equation}
	where $\inf\varnothing:=\infty$. Evidently, the $T_k$'s
	are $\{\F_t\}_{t>0}$-stopping times, where $\F_t$
	denotes the filtration generated by time $t$ by all the
	values of the white noise. Without loss of any generality
	we may assume that $\{\F_t\}_{t>0}$ is augmented in the
	usual way, so that $t\mapsto v_t$ is a $C(\R)$-valued
	strong Markov process.
	
	Next we observe that for every $k\ge 1$,
	\begin{equation}
		\e^k v_{T_k}(x) \ge \mathbf{1}_{(-1,1)}(x)
		\quad\text{for all $x\in\R$, a.s.\ on 
		$\{T_k<\infty\}$.}
	\end{equation}
	Therefore, we apply first the strong Markov property,
	and then Mueller's comparison principle, in order
	to see that the following holds a.s.\ on $\{T_k<t\}$:
	\begin{equation}\begin{split}
		&\P\left( \left. T_{k+1}-T_k < \frac{t}{2n}\
			\right|\, \mathcal{F}_{T_k}\right)\\
		&\hskip1in\le \P\left\{ \inf_{
			s\in(0,t/(2n))}  \inf_{x\in(-1,1)}
			v^{(k+1)}_s(x) \le \e^{-k-1}\right\},
	\end{split}\end{equation}
	where $v^{(k+1)}$ is the unique continuous solution to \eqref{heat}
	[for a different white noise, pathwise], starting at $v^{(k+1)}_0(x):=
	\exp(-k)\1_{(-1,1)}(x)$. Note that
	\begin{equation}
		w^{(k+1)}_t(x) := \e^k v^{(k+1)}_t(x)
	\end{equation}
	solves the SPDE
	\begin{equation}
		\frac{\partial}{\partial t} w^{(k+1)}_t(x) = \frac12 \frac{%
		\partial^2}{\partial x^2}w^{(k+1)}_t(x) + \sigma_k\!\left(
		w^{(k+1)}_t(x)\right)\eta^{(k+1)}_t(x),
	\end{equation}
	subject to $w^{(k+1)}_0(x)=\mathbf{1}_{(-1,1)}(x)$,
	where $\eta^{(k+1)}$ is a space-time white noise for every
	$k$ and
	\begin{equation}
		\sigma_k(x) := \e^k \sigma\!\left( \e^{-k}x\right).
	\end{equation}
	Therefore, the following holds a.s.\ on $\{T_k<t\}$:
	\begin{equation}\label{key}\begin{split}
		&\P\left( \left. T_{k+1}-T_k < \frac{t}{2n}\
			\right|\, \mathcal{F}_{T_k}\right)\\
		&\hskip1in\le \P\left\{ \sup_{\substack{x\in(-1,1)\\
			s\in(0,t/(2n))}} \left| w^{(k+1)}_s(x) -w^{(k+1)}_0(x)
			\right| \ge 1-\frac{1}{\e}\right\}.
	\end{split}\end{equation}
	
	Let $\lip$ denote the optimal Lipschitz constant of $\sigma$.
	Because $\sigma(0)=0$, it follows that
	\begin{equation}
		\sup_{k\ge 1} \left| \sigma_k(z) \right| \le \lip|z|
		\qquad\text{for all $z\in\R$}.
	\end{equation}
	It is this important property that allows us
	to appeal to the estimates of \cite[Appendix]{FK} ,
	and deduce the following: For all $\eta\in(0\,,1)$,
	there exists a constant $Q:=Q(\eta)\in(0\,,\infty)$ such that for all $k\ge 0$,
	$m\in[2\,,\infty)$, and $\tau\in(0\,,1)$,
	\begin{equation}
		\sup_{k\ge 0}
		\E\left( \sup_{\substack{x\in(-1,1)\\s\in (0,\tau)}}
		\left| \frac{w^{(k+1)}_s(x) -w^{(k+1)}_0(x)}{s^{\eta/4}}
		\right|^m \right) \le Q \e^{Qm^3\tau}.
	\end{equation}
	In other words,
	\begin{equation}
		\sup_{k\ge 0}
		\E\left( \sup_{\substack{x\in(-1,1)\\s\in (0,\tau)}}
		\left| w^{(k+1)}_s(x) -w^{(k+1)}_0(x)
		\right|^m \right) \le Q \e^{Qm^3\tau} \tau^{\eta m/4}.
	\end{equation}
	We apply this inequality with $\tau:=t/(2n)$ and optimize over $m$ in order to 
	deduce from \eqref{key} that there exists a constant $L:=L(\eta,t)\in(0\,,\infty)$
	such that for all integers $n>t/2$
	the following holds a.s.\ on $\{T_k<t\}$:
	\begin{equation}\label{key2}
		\P\left( \left. T_{k+1}-T_k < \frac{t}{2n}\
		\right|\, \mathcal{F}_{T_k}\right)
		\le L\exp\left( -L n^{\nicefrac12} (\log n)^{\nicefrac32}\right).
	\end{equation}
	Finally, we notice that if $T_n<t$, then certainly there are
	at least $\lfloor n/2\rfloor$-many distinct 
	values of $k\in\{0\,,\ldots,n-1\}$ such that
	$T_{k+1}-T_k\le t/(2n)$. Therefore, \eqref{key2} implies
	that for all $n>t/2$,
	\begin{equation}
		\P\left\{ T_n<t\right\} \le \binom{n}{\lfloor n/2\rfloor}
		L^{\lfloor n/2\rfloor}
		\exp\left( -L \left\lfloor n/2\right\rfloor n^{\nicefrac12}
		(\log n)^{\nicefrac32}\right).
	\end{equation}
	This, Stirling's formula, and monotonicity together imply 
	\eqref{goal:v}.
\end{proof}

\begin{small}
\bigskip

\noindent\textbf{Daniel Conus}\\
\noindent Lehigh University, Department 
	of Mathematics, Christmas--Saucon Hall, 14 East Packer Avenue,
	Bethlehem, PA 18015 (\texttt{daniel.conus@lehigh.edu})\\[2mm]
\noindent\textbf{Mathew Joseph} \&\  \textbf{Davar Khoshnevisan}\\
\noindent 155 South 1400 East, University of Utah,
	Department of Mathematics, 
	Salt Lake City, UT 84112-0090 (\texttt{joseph@math.utah.edu} \&\
	\texttt{davar@math.utah.edu})
\end{small}


\begin{thebibliography}{999}

\bibitem{Burkholder} Burkholder, D. L.,
	Distribution function inequalities for martingales, 
	{\it Ann.\ Probab.}\ {\bf 1} (1973) 19--42.
%
\bibitem{CJK} Conus, Daniel, Joseph, Mathew, and Khoshnevisan, Davar,
	On the chaotic character of the stochastic heat equation, before
	the onset of intermittency,
	{\it  Ann.\ Probab.} (2011, to appear). \\
	Available electronically at 
	\url{http://arxiv.org/abs/1004.2744}.
%
\bibitem{Dalang:99} Dalang, Robert C.,
	Extending the martingale measure stochastic integral with
	applications to spatially homogeneous s.p.d.e.'s,
	{\it Electron. J. Probab.}\ {\bf 4}, Paper no.\ 6 (1999) 29 pp.\
	(electronic). \\
	Available electronically at
	\url{http://www.math.washington.edu/~ejpecp}.
%
\bibitem{ErdosRenyi} Erd\H{o}s, Paul and Alfr\'ed R\'enyi,
	On a new law of large numbers, {\it J. Analyse Math.}\
	{\bf 23} (1970) 103--111.
%
\bibitem{FK} Foondun, Mohammud and Khoshnevisan, Davar,
	Intermittence and nonlinear parabolic stochastic partial
	differential equations, {\it Electr.\ J. Probab.}\ {\bf 14}\, Paper no.\ 21,
	548--568. \\
	Available electronically at 
	\url{http://www.math.washington.edu/~ejpecp}.
%
\bibitem{HallHeyde} Hall, P. and C. C. Heyde,
	{\it Martingale Limit Theory and Its Application},
	Academic Press, New York, 1980.
%
\bibitem{KPZ} Kardar, Mehran, Giorgio Parisi, and Yi-Cheng Zhang,
	Dynamic scaling of growing interfaces, {\it Phys.\ Rev.\ Lett.}\ 
	{\bf 56}{\it (9)}   (1985) 889--892.
%
\bibitem{KZ} Kardar, Mehran and Yi-Cheng Zhang,
	Scaling of directed polymers in random media,
	\emph{Phys.\ Rev.\ Lett.}\ {\bf 58}{\it (20)}  (1987) 2087--2090.
%
\bibitem{Mueller} Mueller, Carl
	On the support of solutions to the heat equation with noise, 
	\emph{Stochastics and Stochastics Rep.}\ 
	\textbf{37}{\it (4)} (1991) 225--245.
%
\bibitem{MN} Mueller, Carl and David Nualart,
	Regularity of the density for the stochastic heat equation,
	{\it Electr.\ J. Probab.}\ {\bf 13} (2008) Paper no.\ 74,
	2248--2258. \\
	Available electronically at 
	\url{http://www.math.washington.edu/~ejpecp}
%
\bibitem{Walsh} Walsh, John B.,
   \emph{An Introduction to Stochastic Partial Differential Equations},
   in: \'Ecole d'\'et\'e de probabilit\'es de Saint-Flour, XIV---1984,
   265--439,
   Lecture Notes in Math., vol.\ 1180, Springer, Berlin, 1986.
%
\end{thebibliography}
\end{document}